\documentclass{amsart}
\usepackage{amssymb, amscd, amsmath, amscd, tikz}

\title{Variation of the canonical height for a family of polynomials}
\author{Patrick Ingram}
\address{Department of Pure Mathematics, University of Waterloo}
\email{pingram@math.uwaterloo.ca}
\date{March 22, 2010; minor changes April 23, 2011}
\thanks{The author's research is supported by a Discovery Grant from NSERC of Canada.}

\newcommand{\QQ}{\mathbb{Q}}
\newcommand{\ZZ}{\mathbb{Z}}
\newcommand{\CC}{\mathbb{C}}
\newcommand{\RR}{\mathbb{R}}

\newcommand{\PP}{\mathbb{P}}

\renewcommand{\AA}{\mathbb{A}}

\newcommand{\Ocal}{\mathcal{O}}

\newcommand{\basin}{\mathcal{B}}
\newcommand{\adele}{\mathbf{A}}
\newcommand{\Supp}{\operatorname{Supp}}
\newcommand{\Spec}{\operatorname{Spec}}
\newcommand{\Gal}{\operatorname{Gal}}

\newcommand{\ord}{\operatorname{ord}}
\newcommand{\MOD}[1]{~(\textup{mod}~#1)}
\newcommand{\h}{\hat{h}}
\newcommand{\hl}{\hat{\lambda}}

\newcommand{\Pic}{\operatorname{Pic}}

\newcommand{\disc}[3]{\mathbb{D}_{#1}(#2; #3)}
\newcommand{\ann}[4]{\mathbb{D}_{#1}(#2; #3, #4)}

\newcommand{\G}{\mathcal{G}}

\newcommand{\pf}{\mathfrak{p}}
\newcommand{\mf}{\mathfrak{m}}

\renewcommand{\epsilon}{\varepsilon}

\newcommand{\Div}{\operatorname{Div}}

\newcommand{\llbracket}{[\hspace{-1.5pt}[}
\newcommand{\rrbracket}{]\hspace{-1.5pt}]}
\newcommand{\llp}{(\hspace{-1.5pt}(}
\newcommand{\rrp}{)\hspace{-1.5pt})}

\newtheorem{theorem}{Theorem}

\newtheorem{lemma}[theorem]{Lemma}

\newtheorem{corollary}[theorem]{Corollary}

\theoremstyle{remark}
\newtheorem*{remark}{Remark}

\begin{document}
\begin{abstract}
A theorem of Tate asserts that, for an elliptic surface $E\rightarrow X$ defined over a number field $k$, and a section $P:X\rightarrow E$, there exists a divisor $D=D(E, P)\in\Pic(X)\otimes \QQ$ such that 
\[\h_{E_t}(P_t)=h_D(t)+O(1),\]
where $\h_{E_t}$ is the N\'{e}ron-Tate height on the fibre above $t$.
  We prove the analogous statement for a one-parameter family of polynomial dynamical systems.  Moreover, we compare, at each place of $k$, the local canonical height with the local contribution to $h_D$, and show that the difference is analytic near the support of $D$, a result which is analogous to results of Silverman in the elliptic surface context.
\end{abstract}

\maketitle

\section{Introduction}

Let $k$ be a number field, $X$ a smooth, projective curve over $k$, and $E$ an elliptic curve over the function field $K=k(X)$ with associated N\'{e}ron-Tate height $\h_E$.  If $E$ has good reduction at $t\in X(\overline{k})$, then the fibre $E_t$ is an elliptic curve over $\overline{k}$ with an associated N\'{e}ron-Tate height $\h_{E_t}$.  Given a point $P\in E(K)$, it is natural to ask how the height $\h_{E_t}(P_t)$ varies as a function of the parameter.  If $h$ is a height on $X$ with respect to a divisor of degree 1, then a result of Silverman \cite{silverman} shows that
\begin{equation}\label{silverman}\h_{E_t}(P_t)=\h_E(P)h(t)+o(h(t)),\end{equation}
where $o(h(t))/h(t)\rightarrow 0$ as $h(t)\rightarrow\infty$.  This was  improved by Tate \cite{tate}, who showed that, for some divisor $D\in\Pic(X)\otimes\QQ$, of degree $\h_E(P)$, we have
\begin{equation*}\h_{E_t}(P_t)=h_D(t)+O(1).\end{equation*}
In particular, if $X=\PP^1$, then the error term in \eqref{silverman} can be replaced with something bounded by an absolute constant (depending on $E$ and $P$), while in general Silverman's bound is improved to $O(h(t)^{\frac{1}{2}})$.

Now, let $f\in K(z)$ be a rational function, and $P\in\PP^1(K)$.  There is, associated to $f$, a canonical height $\h_f:\PP^1(\overline{K})\rightarrow\RR$ determined uniquely by the properties
\[\h_f(f(P))=\deg(f)\h_f(P)\qquad\text{and}\qquad\h_f(P)=h(P)+O(1),\]
 and similarly to each specialization $f_t$ at which $f$ has good reduction.
  The analogue of \eqref{silverman} holds again here; Call and Silverman~\cite[Theorem~4.1]{call-silv} have shown that 
\begin{equation}\label{eq:call-silverman}\h_{f_t}(P_t)=\h_f(P)h(t)+o(h(t)).\end{equation}
It is natural to ask if the analogue of Tate's theorem holds in this context.
We show that it does, when $f$ is a polynomial.
\begin{theorem}\label{th:main}
Let $k$, $X$, and $K$ be as above, let $f\in K[z]$, and let $P\in\PP^1(K)$.  Then there is a divisor $D=D(f, P)\in\Pic(X)\otimes\QQ$ of degree $\h_f(P)$ such that
\begin{equation}\label{eq:main bound}\h_{f_t}(P_t)=h_{D}(t)+O(1),\end{equation}
as $t\in X(\overline{k})$ varies, where the implied constant depends only on $f$ and $P$.
 \end{theorem}
  The divisor $D(f, P)$ is not hard to define: identifying elements of $K$ with morphisms $X\to \PP^1$, and associating to these the usual pull-back maps from $\Pic(\PP^1)$ to $\Pic(X)$, we may take
 \[D(f, P)=\lim_{N\to\infty}d^{-N}f^N(P)^*(\infty).\]

One immediate application of Theorem~\ref{th:main} is that it allows one to count points on the base for which $\h_{f_t}(P_t)$ is less than a given bound.  It follows from the result of Call and Silverman that for any $B$ and $d$, the quantity
\[N_{f, P}(B, D)=\#\left\{t\in X(\overline{k}):[k(t):k]\leq d\text{ and }\h_{f_t}(P_t)\leq B\right\}\]
 is finite, so long as $\h_f(P)\neq 0$, but nothing stronger than finiteness follows from~\eqref{eq:call-silverman}.  In the case $X=\PP^1$, Theorem~\ref{th:main} combined with a result of Schanuel \cite{schanuel} allows one to deduce that
\[N_{f, P}(B, d)\gg\ll e^{2Bd/\h_f(P)},\]
where the implied constanta depend on $k$, $d$, $f$, and $P$.

Theorem~\ref{th:main} also leads to an improved error term in \eqref{eq:call-silverman}, using an observation due to Lang.
\begin{corollary}\label{cor:main}
Let $k$, $X$, $K$, $f$, and $P$ be as above.  If $h$ is any height on $X$ relative to a divisor of degree 1, then for $t\in X(\overline{k})$ we have
\[\h_{f_t}(P_t)=\h_f(P)h(t)+O\left(h(t)^{\frac{1}{2}}\right),\]
as $h(t)\to\infty$, where the implied constant depends only on $f$ and $P$.
If $X=\PP^1$, then we have the further improvement
\[\h_{f_t}(P_t)=\h_f(P)h(t)+O(1).\]
\end{corollary}

It is perhaps somewhat surprising that an analogue of Tate's theorem can be derived in this context.  The proof of Tate's result relies heavily on both the N\'{e}ron model and group structure of elliptic curves.  Neither of those tools are available in the context of dynamics.  Call and Silverman \cite{call-silv} introduced a notion of weak N\'{e}ron models, which one might hope would help in this context, but Hsia \cite{hsia} has shown, over local fields, that these sometimes fail exist.  Indeed, in the present context, the situation is somewhat more dire.  If a given rational function $f(z)\in K(z)$ admits a weak N\'{e}ron model at every place, then by Theorem~3.1 of \cite{hsia}, the multipliers of the periodic cycles are integral at every place, and hence constant.  In other words, if $\mathcal{M}_d$ is the moduli space of rational functions of degree $d$, and one considers the map $F:X\to \mathcal{M}_d$ the generic fibre of which is $f$, and $\Lambda_N:\mathcal{M}_d\to \AA^m$ is the map taking a rational function to the symmetric functions in the multipliers of its points of period dividing $N$, we have that $\Lambda_N\circ F$ is constant.  A result of McMullen \cite{mcmullen} shows that the map $\Lambda_N$ is finite-to-one, for $N$ large enough, except on Latt\`{e}s maps, and so we have shown that for $f$ to admit a weak N\'{e}ron model at every place, $f$ must either be isotrivial, or a family of Latt\'{e}s maps (i.e., a family coming from an elliptic surface, and hence to which Tate's result applies).
  In light of this,
it would be particularly interesting if one could extend Theorem~\ref{th:main} to apply to all rational functions.  If such a result could be shown, this would give a proof of Tate's theorem which makes no fundamental use of the N\'{e}ron model or the group structure of an elliptic curve, via the machinery of Latt\`{e}s maps.

Tate's results in \cite{tate} are not the end of the story for the variation of canonical heights on elliptic surfaces.  Silverman \cite{var1, var2, var3} showed that the difference $\h_{E_t}(P_t)-h_D(t)$, in addition to being bounded, varies quite regularly as a function of $t$, breaking up into a finite sum of well-behaved functions at various places of $k$.  For example, if
\[E_t:y^2=x^3+t^2(1-t^2)x\quad\text{and}\quad P_t=(t^2, t^2),\]
then the first result of \cite{var1} shows that there is a real-analytic function $F(x)$ defined on a neighbourhood of $0$, such that $F(0)=0$ and, for all $t\in\ZZ$ sufficiently large,
\[\h_{E_t}(P_t)=\h_{E}(P)h(t)+\frac{1}{4}\log 2+F\left(\frac{1}{t^2}\right).\]
%
In the present context, we may also derive results analogous to those of \cite{var1, var2, var3}.

For example, let $k=\QQ$,  $X=\PP^1$, $f_t(z)=z^2+t$, and $P_t=0$.  One can show that  for $t\in\QQ$ in a (real) neighbourhood of infinity, 
\begin{equation}\label{eq:explicit series}
\h_{f_t}(P_t)=\h_f(P)h(t)+\frac{1}{4t}-\frac{1}{8t^2}+\frac{5}{24t^3}-\frac{5}{16t^4}+\frac{17}{40t^5}-\frac{29}{48t^6}+\cdots\end{equation}
(where in this case $\h_f(P)=\frac{1}{2}$, and $h$ is the usual Weil height on $\PP^1$).  
More generally, we derive the following result for quadratic polynomials over $\QQ(t)$.
\begin{theorem}\label{th:quad polys}
Let $f_t(z)=z^2+t$, and let $P_t\in\ZZ[t]$ be a monic polynomial.  Then there exists a  function $F(z)\in \QQ\llbracket z\rrbracket$, convergent in a (real) neighbourhood of $0$ and satisfying $F(0)=0$, such that for all $t\in\QQ$ with $|t|$ sufficiently large,
\[\h_{f_t}(P_t)=\h_f(P)h(t)+F\left(\frac{1}{t}\right).\]
\end{theorem}
Theorem~\ref{th:quad polys} is essentially a special case of a more general theorem, which is analogous to the results of Silverman \cite{var1, var2, var3}.  Roughly speaking, the theorem below says that the difference between $\h_{f_t}(P_t)$ and $h_D(t)$ is given by a sum of real-analytic functions, so long as $t$ is close enough to $\Supp(D)$, on some prescribed set of places of $k$.  The statement of the result is somewhat more involved, however, since the analytic functions depend on which point in $\Supp(D)$ is approached by $t$ at each place.  It should be noted that, since the points in $\Supp(D)$ need not be $k$-rational, the following theorem assumes that we have fixed an extension of each valuation on $k$ to a valuation on $\overline{k}$.  It should also be noted that the height $h_D$ below is a particular height function, although it will be clear from the proof that one can adjust the terms involved to accommodate any suitably well-behaved height.
\begin{theorem}\label{th:variation}
Let $k$,  $X$, $f$, and $P$ be as above.   Then there exists 
 a finite set of places $S\subseteq M_k$, containing all infinite places;
 for each pair $\beta\in \Supp(D)$ and $v\in S$ a neighbourhood $U_{\beta, v}\subseteq X(\overline{k}_v)$ of $\beta$; and
 for each pair $\beta\in \Supp(D)$ and $v\in S$ archimedean, a function $F_{\beta, v}:U_{\beta, v}\to \RR$ which is  real-analytic, with $F_{\beta, v}(\beta)=0$,
 such that for any  $\phi:S\to\Supp(D)$ there exists a $C(\phi)\in\RR$  such that 
\begin{equation*}
\h_{f_t}(P_t)=h_D(t)+C(\phi)+\sum_{v\mid\infty}F_{\phi(v), v}(t)
\end{equation*}
for any $t\in X(k)$ satisfying
 $t\in U_{\phi(v), v}$ for all $v\in S$.
In particular,
\[\h_{f_t}(P_t)=h_D(t)+ C(\phi)+o(1),\]
where $o(1)\to 0$ as $t\to\phi(v)\in\Supp(D)$ in the $v$-adic topology, simultaneously for all $v\in S$.
\end{theorem}

\begin{remark}
We will, in fact, prove something stronger.  It turns out that our maps $F_{\beta, v}$ are of the form
\[F_{\beta, v}(t)=\frac{[k_v:\QQ_v]}{[k:\QQ]}\log\left|\tilde{F}_\beta(t)\right|_v\]
for some $\tilde{F}_\beta \in \widehat{\Ocal}_{\beta, X}$, where $\Ocal_{\beta, X}$ is the local ring of $X/k$ at $\beta$, and  $\widehat{\Ocal}_{\beta, X}$ its completion in the local topology.  This \emph{a priori} formal function $\tilde{F}_\beta$ turns out to be $v$-adic analytic at $\beta$ for all $v\in M_k$.    This is noticeably stronger than the statement of Theorem~\ref{th:variation}, as it shows that the real-analytic functions $F_{\beta, v}$ arise from more fundamental analytic functions which depend only on the $\beta\in\Supp(D)$.  It also shows that the power series defining the functions $F_{\beta, v}$ have coefficients in some finite extension of $k$.  Similarly, the constants $C(\phi)$ turn out to have the form
\[C(\phi)=d^{-N}(d-1)^{-1}\sum_{v\in S}\frac{[k_v:\QQ_v]}{[k:\QQ]}\log |c_{\phi(v)}|_v,\]
for some $N\geq 0$, and some values $c_\beta\in k^*$ indexed by $\beta\in \Supp(D)$, which are $v$-units for any $v\not\in S$.  In particular, it follows from the product formula that $C(\phi)=0$ if $\phi$ is constant.
\end{remark}

\begin{remark}
The proof of Theorem~\ref{th:variation} is easily modified to give a similar result for points $t\in X(\overline{k})$,  and we present that (somewhat more complicated) statement below.
Indeed, since the proof of this theorem turns out to be purely local, we could replace the $k$-rational points on $X$ with the points rational over the adele ring $\adele_{\overline{k}}$, and similarly for Theorem~\ref{th:main}.
\end{remark}

Before proceeding, we consider a slightly more revealing example of Theorem~\ref{th:variation}. If $f_t(z)=z^2+t$ and $P_t=7t+t^{-1}$, then our definition above gives $D(f, P)=(0)+(\infty)$.  Let $v_\infty$ and $v_7$, respectively, denote the archimedean and $7$-adic valuations on $\QQ$, and let $\phi:S=\{v_\infty, v_7\}\to\Supp(D)=\{0, \infty\}$.
From the proof of Theorem~\ref{th:variation}, we see that for  $t\in \PP^1(\QQ)$, we have
\[\h_{f_t}(P_t)=h_D(t)+\log|c_{\phi(v_\infty)}|_\infty+\log|c_{\phi(v_7)}|_7+o(1)\]
where $o(1)\to 0$ as $t\to \phi(v)$ in the $v$-adic topologies.
It turns out, in this case, that $c_{\infty}=7$ and $c_{0}=1$.
In particular, as $t\to \infty$ at the archimedean place, and $t\to 0$ at the $7$-adic place, we have
\[\h_{f_t}(P_t)=h_D(t)+\log 7+o(1).\]
In contrast, as $t\to\infty$ in both topologies, we have
\[\h_{f_t}(P_t)=h_D(t)+o(1).\]


\section{Local and global heights}

To begin, we will set down some notation and preliminary results.  Most of the terminology is standard, and can be found, for example, in \cite{call-silv}, \cite{lang}, and \cite{js:ads}, but we will recall the basic notation here.
First of all, let $L$ be a field, and let $v$ be a valuation on $L$.  Then for any polynomial $f(z)\in L[z]$ of degree $d\geq 2$, we define a \emph{local canonical height}
\begin{equation}\label{eq:local height}\hl_{f, v}(z)=\lim_{N\to\infty}d^{-N}\max\{0, \log|f^N(z)|_v\}.\end{equation}
It is perhaps not immediately clear that this limit exists for all $z\in L$.
If 
\[f(z)=a_dz^d+a_{d-1}z^{d-1}+\cdots +a_1z+a_0,\]
with $a_i\in L$ and $a_d\neq 0$, let 
\[\basin_v(f)=\left\{z:|f^N(z)|_v\to\infty\text{ as }N\to\infty\right\}\]
denote the $v$-adic basin of infinity.
Furthermore, let the symbol $(2d)_v$ denote $2d$ if $v$ is archimedean, and 1 otherwise, let 
\[\Lambda_v=\max\left\{\max_{0\leq i<d}\left\{\left|\frac{a_i}{a_d}\right|_v^{1/(d-i)}\right\}, |a_d|_v^{-2/(d-1)},  1\right\},\]
and let
\[\basin_v^0(f)=\left\{z: |z|_v>(2d)_v\Lambda_v \right\}.\]
Note that it is perhaps more natural to replace $|a_d|_v^{-2/(d-1)}$, in the above definition, with $|a_d|_v^{-1/(d-1)}$, but the more restrictive bound is critical later.

Roughly speaking, $\basin_v^0(f)$ will play the r\^{o}le played in the elliptic curve context by $\mathcal{E}^0_v$, the identity component of the fibre of the N\'{e}ron model above $v$.  In other words, $\basin_v^0(f)$ is some domain on which the local heights are particularly well-behaved. 
The following elementary results describe the behaviour of local heights in $\basin_v^0(f)$; similar results appear in \cite{baker-hsia} and \cite{pi-lang}.
\begin{lemma}\label{lem:basins}
For all $z\in L$, the limit defining $\hl_{f, v}(z)\geq 0$ exists, and $\hl_{f,v}(z)>0$ if and only if $z\in\basin_v(f)$.  Furthermore, both $\basin_v(f)$ and $\basin_v^0(f)$ are closed under $f$, and \[\basin_v(f)=\left\{z:f^N(z)\in\basin_v^0(f)\text{ for some }N\geq 0\right\}.\]  Finally, for all $z\in \basin_v^0(f)$ and all $N$, we have
\[c_1\leq \frac{1}{d^N}\log\left|f^N(z)\right|_v-\left(\frac{1-d^{-N}}{d-1}\log|a_d|_v+\log|z|_v\right)\leq c_2,\]
where $c_1=\log\frac{1}{2}$ and $c_2=\log\frac{3}{2}$ if $v$ is archimedean, and $c_1=c_2=0$ otherwise.
In particular,
\[c_1\leq \hl_{f, v}(z)-\left(\frac{1}{d-1}\log |a_d|_v+\log|z|_v\right)\leq c_2.\]
\end{lemma}

\begin{proof}
Let $d=\deg(f)$.
First, we note that for $z\in \basin_v^0(f)$, we have
by hypothesis, 
\begin{equation}\label{eq:basin inequality}(2d)_v|a_i|_v|z|^i_v\leq (2d)_v^{(d-i)}|a_i|_v|z|_v^i< |a_d|_v|z|_v^d.\end{equation}
If $v$ is non-archimedean, this implies
\begin{equation}\label{eq:nonarch basin}|f(z)|_v=|a_d|_v|z|_v^d\geq |z|_v,\end{equation}
whereupon $f(z)\in\basin_v^0(z)$.  On the other hand, if $v$ is archimedean, we have
\[\left|f(z)\right|_v=\left|\sum a_iz^i\right|_v\geq |a_d|_v |z|^d-d\max |a_i|_v|z|^i_v\geq \frac{1}{2}|a_d|_v |z|_v^d\]
by~\eqref{eq:basin inequality}.  It follows again that $|f(z)|_v\geq |z|_v$, and hence $f(z)\in \basin_v^0(f)$.  Thus, in either case, $\basin_f^0(f)$ is closed under $f$.

Now, if $v$ is non-archimedean, then \eqref{eq:nonarch basin} implies
\[|f^N(z)|_v=|a_d|_v^{(d^N-1)/(d-1)}|z|^{d^N}_v\]
by induction, for all $z\in\basin_v^0(f)$.  Since $|z|_v> 1$, we obtain
\[\hl_{f, v}(z)=\lim_{N\to\infty}d^{-N}\log\left(|a_d|_v^{(d^N-1)/(d-1)}|z|^{d^N}_v\right)=\frac{1}{d-1}\log|a_d|_v+\log|z|_v.\]
If, on the other hand, $v$ is archimedean, then \eqref{eq:basin inequality} gives
\[\frac{1}{2}|a_d|_v|z|_v^d\leq |f(z)|_v\leq \frac{3}{2}|a_d|_v|z|_v^d\]
for all $z\in\basin_v^0(f)$, and so by induction,
\[\left(\frac{1}{2}|a_d|_v\right)^{\frac{d^N-1}{d-1}} |z|_v^{d^N}\leq |f^N(z)|_v\leq \left(\frac{3}{2}|a_d|_v\right)^{\frac{d^N-1}{d-1}} |z|_v^{d^N}.\]
Taking logarithms and  limits yields
\[\frac{1}{d-1}\log\frac{1}{2}\leq\hl_{f, v}(z)-\left(\frac{1}{d-1}\log|a_d|_v+\log|z|_v\right) \leq \frac{1}{d-1}\log\frac{3}{2}\]
which, in the worst case $d=2$, is what was claimed.

Now, if $z\not\in \basin_v(f)$, then $|f^N(z)|_v$ is bounded as $N\to\infty$, and so $\hl_{f, v}(z)=0$.  On the other hand, if $z\in\basin_v(f)$
then there is some $N$ with $|f^N(z)|_v>(2d)_v\Lambda_v$, and so we have both $f^N(z)\in \basin_v^0(f)$, and $\hl_{f, v}(z)>0$.
\end{proof}

We now recall the definition of various height functions.
Throughout, $k$ will denote some number field, and $M_k$ will be the standard set of places on $k$.  We will adopt the convention that the valuation $|\cdot|_v$, for each $v\in M_k$, has been extended in some way to $\overline{k}$.  For each $v\in M_k$, we define a local (na\"{i}ve) height on $\PP^1$ by
\[\lambda_v(x)=\max\{0, \log|x|_v\}.\]
The global (na\"{i}ve) height on $\PP^1(k)$ is defined by
\[h(x)=\sum_{v\in M_k}\frac{[k_v:\QQ_v]}{[k:\QQ]}\lambda_v(x).\]
It is easy enough to see that this can be extended to $\overline{k}$ by defining
\[h(x)=\sum_{v\in M_k}\frac{[k_v:\QQ_v]}{[k:\QQ]}\left(\frac{1}{[L:k]}\sum_{\sigma\in\Gal(L/k)}\lambda_v(x^\sigma)\right),\]
where $L\supseteq k$ is any Galois extension containing $x$.
It is, of course, necessary to check that this definition does not depend on the particular Galois extension chosen, but it does not.
We define the canonical height with respect to $f\in k[z]$ by
\[\h_f(x)=\sum_{v\in M_k}\frac{[k_v:\QQ_v]}{[k:\QQ]}\hl_{f, v}(x),\]
and similarly for $\overline{k}$.

We define heights in function fields similarly, although we work over the algebraic closure of the constant field (so that valuations on $K=k(X)$ are the same as valuations on $K\otimes\overline{k}$).  Also, we will denote the valuation corresponding to $\beta\in X(\overline{k})$ by $\ord_\beta$ to avoid confusion with valuations on the constant field $k$.
 For any $z\in K$ and $\beta\in X(\overline{k})$, we define
\[\lambda_\beta(z)=\max\{0, -\ord_\beta(z)\},\]
so that $\lambda_\beta(z)$ is the order of the pole of $z$ at $\beta$, if there is one, and 0 otherwise. For $f\in K[z]$ we define $\hl_{f, \beta}$ as in \eqref{eq:local height}, and set
\[\h_{f}(z)=\sum_{\beta\in X(\overline{k})}\hl_{f, \beta}(z).\]
At this point we can define our divisor $D=D(f, P)\in \Pic(X)\otimes\QQ$, which will simply be
\begin{equation}\label{eq:divisor definition}D(f, P)=\sum_{\beta\in X(\overline{k})}\hl_{f, \beta}(P)(\beta).\end{equation}
This clearly has degree $\h_f(P)$, and is equivalent to the definition of $D(f, P)$ given in the introduction.

In addition to the above heights on $\PP^1$, we will define N\'{e}ron functions, and heights relative to divisors on $X$.  Since we want to claim that the difference $\hl_{f_t, v}(P_t)-\lambda_{D, v}(t)$ is real-analytic in certain neighbourhoods, we need to be fairly specific as to how we define these local heights.  Let $D=D(f, P)$ be as defined above, for a particular $f\in K[z]$ and $P\in K$.  If it should happen that $D=0$, then we will simply define $\lambda_{D, v}(x)=0$ for all $v\in M_k$ and $x\in X(\overline{k}_v)$.  To deal with the case $D\neq 0$, we will employ the following simple lemma.

\begin{lemma}\label{lem:r-r}
With $f$ and $P$ as above, suppose that $D=D(f, P)\neq 0$.   Then there is an $N$ and a morphism $g:X\to\PP^1$ (defined over $k$) such that $d^N(d-1)D=g^*(\infty)$.
\end{lemma} 

\begin{proof}  First we must show that there is an $N$ with $d^N(d-1)D\in \Div(X)$, under the hypothesis that $D>0$.
For each $\beta\in \Supp(D)$, we have $\hl_{f, \beta}(P)>0$, and so by Lemma~\ref{lem:basins} there is an $N$ such that $f^N(P)\in\basin_v^0(f)$.  For this value of $N$, we have
\[d^N(d-1)\hl_{f, \beta}(P)=(d-1)\hl_{f, \beta}(f^N(P))=(d-1)\log|f^N(P)|_\beta+\log|a_d|_\beta\in\ZZ.\]
  If we choose $N$ large enough that $f^N(P)\in \basin_\beta^0(f)$ for all $\beta\in\Supp(D)$, we have
  \[d^N(d-1)D=\sum_{\beta\in X(\overline{k})}d^N(d-1)\hl_{f, \beta}(P)(\beta)\in \Div(X).\]
 Now, since $D>0$, we may choose $N$ to be large enough so that \[d^N(d-1)\deg(D)\geq 2g(X)\] which ensures, by the Riemann-Roch theorem, that there is a morphism $g:X\to\PP^1$ such that $d^N(d-1)D=g^*(\infty)$.
\end{proof}

Thus in the case $D\neq 0$, we may choose $N$ and $g$ as in Lemma~\ref{lem:r-r}, and set for each $v\in M_k$
\[\lambda_{D, v}(t)=d^{-N}(d-1)^{-1}\max\{0, \log|g(t)|_v\}.\]
The global height is defined by
\[h_{D}(t)=\sum_{v\in M_k}\frac{[k_v:\QQ_v]}{[k:\QQ]}\lambda_{D, v}(t)\]
for $t\not\in\Supp(D)$, and $h_D(t)=0$ for $t\in\Supp(D)$
(and similarly for points $t\in X(\overline{k})$).  By linearity and functoriality of heights (see, e.g., \cite{lang}), the global height $h_D:X(\overline{k})\to\RR$ differs from any other height relative to $D$ by at most a bounded amount.  Everything below transfers over to any suitably well-behaved choice of local heights.

To keep track of bounds which depend on places $v\in M_k$, we will use Weil's notion of an $M_k$-divisor \cite[p.~29]{lang}.  
A \emph{(multiplicative) $M_k$-divisor} is a function $\mathfrak{c}:M_k\rightarrow \RR^+$ such that $\mathfrak{c}(v)=1$ for all but finitely many $v\in M_k$, and such that for each non-archimedean $v$, $\mathfrak{c}(v)=|\alpha|_v$ for some $\alpha\in k^*$.
It is clear that the $M_k$-divisors form a group under pointwise multiplication, and that the pointwise maximum or minimum of two $M_k$-divisors is again an $M_k$-divisor.

Additionally, given a  place $v\in M_k$, a point $\beta\in X(\overline{k})$, and a function $u_\beta\in\overline{k}(X)$,  vanishing only at $\beta$, we will set
\[\disc{v}{\beta}{\epsilon}=\left\{t\in X(\overline{k}_v):|u_\beta(t)|^{1/\ord_\beta(u_\beta)}_v<\epsilon\right\}\]
and
\[\ann{v}{\beta}{\delta}{\epsilon}=\left\{t\in X(\overline{k}_v):\delta<|u_\beta(t)|^{1/\ord_\beta(u_\beta)}_v<\epsilon\right\}.\]
 Note that these sets depend on the choice of $u_\beta$, but for a different choice the set will agree at all but finitely many places.


\section{Analytic properties}

The proofs of Theorem~\ref{th:main} and Theorem~\ref{th:variation} are, not surprisingly, largely analytic in nature.  Here we lay some of the groundwork.  In this section, we will typically take $L$ to be the number field $k$, with $M_L=M_k$ a set of places which have all been extended to $\overline{L}$ in some way.  For purely local results, however, we may also take $L$ to be a any field with a valuation $v$, in which case we take $M_L=\{v\}$, and interpret an $M_L$-divisor as simply a function $\{v\}\to \RR^+\cap v(L)$.  

Throughout, $X/L$ will be a smooth projective curve, and $K=L(X)$ its function field.  For each $\beta\in X(\overline{L})$, we fix a uniformizer $w_\beta\in K$, and let $K_\beta$ denote the completion of $K$ with respect to $\ord_\beta$.  As usual, $\Ocal_{\beta, X}\subseteq K$ will denote the subring consisting of elements regular at $\beta$, and similarly for $\widehat{\Ocal}_{\beta, X}\subseteq K_\beta$.  Note that there is a natural isomorphism $\widehat{\Ocal}_{\beta, X}\cong L\llbracket w_\beta\rrbracket$, and we
will associate elements of $\widehat{\Ocal}_{\beta, X}$ with their series representations, and similarly for $K_\beta$.  The element $g\in\widehat{\Ocal}_{\beta, X}$ is \emph{$v$-adic analytic} if the corresponding series converges on the disk $\disc{v}{\beta}{\epsilon}$, for some $\epsilon>0$, and an element in $K_\beta$ is analytic if it is the quotient of two analytic elements (note that such a function might have a pole at $\beta$).
 Similarly, if  $v\in M_L$ is non-archimedean, then $\Ocal_{v, L}$ will denote the ring of $v$-adic integers, defined by $\{x\in L:|x|_v\leq 1\}$.  
 
 For the remainder of the section, we will fix $f\in K_\beta[z]$ and $P\in K_\beta$, such that $P\in\basin^0_{\ord_\beta}(f)$, and set $0<m=-\ord_\beta(P)$.
The next lemma is (in the case where $L$ is a number field) a well-known adelic version of the implicit function theorem, which we will use to translate the problem into one of pure analysis.  Given a formal power series $F\in L\llbracket w\rrbracket$, we will say that $\epsilon>0$ is a \emph{($v$-adic) radius of convergence} for $F$ if the sum $F(w)$ converges $v$-absolutely for $|w|_v<\epsilon$. We will say that the $M_L$-divisor $\mathfrak{e}$ is a \emph{global radius of convergence} for $F$ if $\mathfrak{e}(v)$ is a $v$-adic radius of convergence for $F$,  for each $v\in M_L$.  Given a Laurent series $F\in L\llp w\rrp$, we will say that $\epsilon$ or $\mathfrak{e}$ is a radius of convergence for $F$ if it is for $w^mF(w)$, where $m$ is the order of the pole of $F$ at $w=0$.

\begin{lemma}\label{lem:convergence}
Let  $g\in \Ocal_{\beta, X}$.  Then $g$ is analytic at every place, i.e.,  there is a (global) radius of convergence for the image of $g$ in $L\llbracket w_\beta\rrbracket$.
\end{lemma}

\begin{proof}
If $v$ is an archimedean place, this is simply the implicit function theorem.  There is a non-archimedean version of the implicit function theorem, and we may apply this at finitely many of the non-archimedean places, but for the conclusion when $L$ is a number field, we need to know that the radius of convergence at $v$ is 1 for all but finitely many $v$.

Let $R=L\llbracket w_\beta\rrbracket$, and let $\pf=w_\beta R$.  If $U\subseteq X$ is an affine subscheme of $X$, containing $\beta$, then $U$ defines a 0-dimensional affine scheme over $R$.  Suppose that  \[U=\Spec(R[X_1, ..., X_s]/(F_1, ..., F_s)).\]
Then, as usual, we may use Newton's Method to lift the point $\beta\in U(R/\pf)$ to a point in $U(R)$.  In other words, if $F(X)$ denotes the vector \[\langle F_1(X_1, ..., X_s), ..., F_s(X_1, ..., X_s)\rangle,\] and $J(F)$ denotes the Jacobian matrix of this system, we let $X_0=\beta$, and take
\[X_{n+1}=X_n-J(F)(X_n)^{-1}F(X_n).\]
Note that $J(F)(X_n)$ is invertible for each $n$, since 
\[\det(J(F)(X_n))\equiv \det(J(F)(\beta))\not\equiv 0\MOD{\pf}.\]
It is simple enough to show that $F(X_n)\in \pf^{2^n}$, for each $n$, and so by the completeness of $R$, 
this sequence of points converges to a limit $Y=(Y_1, ..., Y_s)$ in $R^s$.  This vector satisfies $Y\equiv \beta\MOD{\pf}$ or, viewing the entries as functions of $w_\beta$, $Y(0)=\beta$.  The tuple $Y$ is also the unique element of $R^s$ with this property.  Thus, the series $Y_1, ... , Y_s$, within their radius of convergence, define the coordinate functions $y_1, ..., y_s$ on $U$.

Now, let $S$ be a finite set of places such that for $v\not\in S$, the coefficients of all of equations defining $U$ are $v$-integral, and $\det(J(F))(\beta)$ is a $v$-unit.  Then it is easy to check, by induction, that $X_n$ has $v$-integral coefficients, and so  \[Y_i\in\Ocal_{v, L}\llbracket w_\beta\rrbracket\subseteq L\llbracket w_\beta\rrbracket.\]  Now, if we represent $g$ in these coordinates
\[g(y_1, ..., y_s)=\frac{G_1(y_1, ..., y_s)}{G_2(y_1, ..., y_s)},\]
we see that $g(Y_1, ..., Y_s)$ has $v$-integral coefficients, so long as $G_2(Y_1(0), ..., Y_s(0))$ is a $v$-unit.  If we enlarge $S$ to contain all places for which this fails, we have $g\in\Ocal_{v, L}\llbracket w_\beta\rrbracket$ for all $v\not\in S$, and so $g(w)$ converges for any $|w|_v<1$.
\end{proof}
\begin{remark}
What we have in fact proven, and we will make use of this below, is that there is a finite set $S\subseteq M_L$, containing all archimedean places, such that for $v\not\in S$, the series for $g$ in $L\llbracket w_\beta\rrbracket$ has $v$-integral coefficients, a result essentially due to Eisenstein.
\end{remark}

  The following statement is essentially a continuity claim, which states that a function on a curve can take $v$-adically large values only at points $v$-adically close to its poles.

\begin{lemma}\label{lem:cont}
Let $Z\subseteq X(L)$ be a finite set, and for each $\alpha\in Z$ let $\mathfrak{e}_\alpha$ be an $M_L$-divisor.  Then for any rational function $g$ on $X$ having no poles on $X\setminus Z$, there is an $M_L$ divisor $\mathfrak{d}$ such that for any $t\in X(\overline{L})$ and any $v\in M_L$, $|g(t)|_v > \mathfrak{d}(v)$ implies
$|w_\alpha(t)|_v< \mathfrak{e}_\alpha(v)$ for some $\alpha\in Z$.
\end{lemma}

\begin{proof}
The conclusion of the lemma only gets weaker as $Z$ gets larger, and so we will assume that $Z$ is exactly the set of poles of $g$.
The curve $X$ is smooth and projective, and so by Lemma~2.2 of \cite[p.~85]{lang}, for any functions $f_1, ..., f_n\in \overline{L}(X)$ with no common zero, there is an $M_L$-divisor $\mathfrak{c}$ with
\[\sup_{1\leq i\leq n}|f_i(t)|_v\geq \mathfrak{c}(v)\]
for all $v\in M_L$ and all $t\in X(\overline{L})$.
Let $n_\alpha$ be the order of the pole of $g$ at $\alpha$, and let $f_1=1/g$, $f_{\alpha}=1/(w_{\alpha}^{n_{\alpha}} g)$, for each $\alpha\in Z$.  The zeros of $f_1$ are contained in $Z$, but $\alpha$ is not a zero of $f_{\alpha}$, and so the lemma applies to this collection of functions.  Let $\mathfrak{c}$ be the $M_L$-divisor with the above property, and  choose
\[\mathfrak{d}(v)=\max_{\alpha\in Z}\{\mathfrak{e}_\alpha(v)^{-n_{\alpha}}, 1\}\mathfrak{c}(v)^{-1}.\]
Note that $|g(t)|_v> \mathfrak{d}(v)$ immediately implies $|f_1(t)|_v<\mathfrak{c}(v)$, and so for each $v\in M_L$, there is some $\alpha\in Z$ with $|f_{\alpha}(t)|_v\geq \mathfrak{c}(v)$.  For that particular $\alpha$ and $v$, then, we have
\[|w_{\alpha}(t)|_v^{-n_{\alpha}}=|g(t)f_{\alpha}(t)|_v> \mathfrak{d}(v)\mathfrak{c}(v)= \max_{\alpha'\in Z}\{\mathfrak{e}_{\alpha'}(v)^{-n_{\alpha'}}, 1\}\geq \mathfrak{e}_\alpha(v)^{-n_{\alpha}}.\]
Since $n_{\alpha}\geq 1$, the result is proven.
\end{proof}

  Our next lemma produces a formal object which, in certain isotrivial cases (take $X=\PP^1$, $f\in L[z]$, and $P_t=t$), corresponds to the B\"{o}ttcher coordinate.  Note that the lemma is slightly ambiguous, since there may be several $d^N$th roots of an element of $K_\beta$, but a choice of roots is made in the proof.

\begin{lemma}\label{lem:formal bottcher}
With the notation above, there exists a $\G_\beta \in K_\beta$ such that 
\[\left(f^N(P)a_d^{-(d^N-1)/(d-1)}\right)^{1/d^N}\longrightarrow\ \G_\beta\]
in the $\mf$-adic topology, as $N\to\infty$.  Furthermore, there is a finite set $S\subseteq M_L$ such that for any $v\not\in S$, the image of $\G_\beta$ in $L\llbracket w\rrbracket$ lies in the subring $\Ocal_{v, L}\llbracket w\rrbracket$ consisting of power series with coefficients integral at $v$.
\end{lemma}

\begin{proof}
Since $P\in \basin_\beta^0(f)$, we have for all $N$,
\[|f^N(P)|_\mf=|P|_\mf^{d^N}|a_d|_\mf^{(d^N-1)/(d-1)},\]
by Lemma~\ref{lem:basins}.  If $m=\hl_{f, \beta}(P)=-\ord_\beta(P)$, and $w=w_\beta$ is a uniformizer at $\beta$, let
\[\xi_N=f^N(P)w^{md^N}a_d^{-(d^N-1)/(d-1)}\in \Ocal_{\beta, X}^*.\]
Note that, for all $Q\in \basin_\beta^0(f)$, we have
\[\left|f(Q)-a_dQ^d\right|_\mf=\left|a_{d-1}Q^{d-1}+\cdots+a_1Q+a_0\right|_\mf<|a_dQ^d|_\mf,\]
and so the leading term of the series $f(Q)$ agrees with the leading term of $a_dQ^d$. 
Thus if $P=\alpha w^{-m}+O(w^{1-m})$ and $a_d=\gamma w^n+O(w^{1+n})$, say, we have
\[f^N(P)=\gamma^{(d^N-1)/(d-1)}\alpha^{d^N}w^{q}+O(w^{1+q})\]
for $q=n(d^N-1)/(d-1)-md^N$, and hence
\[\xi_N=\alpha^{d^N}+O(w).\]
Thus the polynomial $\Phi(X)=X^{d^N}-\xi_N$ has a simple root, modulo $\mf$, at $X=\alpha$, and so
by Hensel's Lemma (a special case of the argument used in Lemma~\ref{lem:convergence}), there is some $G_N\in \widehat{\Ocal}_{\beta, X}$ such that $G_N^{d^N}=\xi_N$, and $G_N\equiv \alpha\MOD{\mf}$.  We wish to show that the sequence $G_N$ has a limit in $\widehat{\Ocal}_{\beta, X}$.

Suppose that we set $B=\max_{0\leq i<d}|a_i/a_d|_\mf$, and take $Q\in\basin_\beta^0(f)$.  Then certainly
\[\left|f(Q)-a_dQ^d\right|_\mf = |a_d|_\mf\left|\sum_{i=0}^{d-1}\frac{a_i}{a_d}Q^i\right|_\mf\leq B|a_d|_\mf  |Q|_\mf^{d-1}\]
(recalling that $Q\in \basin_\beta^0(f)$ implies $|Q|_\mf>1$).
Since $\basin_\beta^0(f)$ is closed under $f$, then, we have
\[\left|f^{N+1}(P)-a_df^N(P)^d\right|_\mf\leq B|a_d|_\mf|f^N(P)|_\mf^{d-1}= B|a_d|_\mf^{d^N} |P|_\mf^{d^N(d-1)},\]
by Lemma~\ref{lem:basins}.
It follows that
\begin{eqnarray*}
\left|\xi_{N+1}-\xi_N^d\right|_\mf&=&
|w|_\mf^{md^{N+1}}|a_d|_\mf^{-(d^{N+1}-1)/(d-1)}\left|f^{N+1}(P)-a_d f^N(P)^d\right|_\mf\\
&\leq & |w|_\mf^{md^{N+1}}|a_d|_\mf^{-(d^{N+1}-1)/(d-1)}B|a_d|_\mf^{d^N} |P|_\mf^{d^N(d-1)}\\
&=&B |P|_\mf^{-d^N}|a_d|_\mf^{-(d^N-1)/(d-1)}\\
&\leq& B|P|_\mf^{-d^N}|P|_\mf^{\frac{1}{2}(d^N-1)}\\
&\leq & B|P|_\mf^{-d^N/2}.
\end{eqnarray*}
Now, the roots of $X^{d^{N+1}}-\xi_{N+1}$ are $\zeta^i G_{N+1}$, for $0\leq i< d^{N+1}$, and $\zeta$ a primitive $d^{N+1}$th root of unity.  We have $G_N=\alpha+O(w)$, and the same for $G_{N+1}$, and so for all $i\neq 0$, 
\[\left|G_N-\zeta^iG_{N+1}\right|_\mf=\left|(1-\zeta^i)\alpha+O(w)\right|_\mf=1.\]
It follows that
\begin{eqnarray*}
|G_N-G_{N+1}|_\mf&=&\prod_{i=0}^{d-1}|G_N-\zeta^iG_{N+1}|_\mf\\
&=&\left|G_N^{d^{N+1}}-\xi_{N+1}\right|_\mf\\
&=&\left|\xi_N^d-\xi_{N+1}\right|_\mf\\
&\leq& B|P|_\mf^{-d^N/2}.
\end{eqnarray*}
By the triangle inequality, and the generous estimate $N<d^N$, we have for all $M\geq N$,
\begin{multline*}\left|G_{M}-G_N\right|_\mf\leq \sum_{i=N}^{M-1}B|P|_\mf^{-d^i/2}<B\sum_{i=N}^{M-1}|P|_\mf^{-i/2} \\< B\int_{N-1}^\infty |P|_\mf^{-s/2}ds=\frac{B|P|_\mf^{-(N-1)/2}}{\frac{1}{2}\log|P|_\mf}.\end{multline*}
Since $P\in \basin_\beta^0(F)$, we have $B<|P|_\mf^d$, and it follows that the sequence of $G_N$ is Cauchy.
Since $\widehat{\Ocal}_{\beta, X}$ is complete with respect to the $\mf$-adic metric, there is a limit $G_\infty\in\widehat{\Ocal}_{\beta, X}^*$ of this sequence.  
We can simply take $\G_\beta=w^{-m}G_\infty$.

It remains to show that, in the case where $L$ is a global field, we may find a finite set of places $S\subseteq M_L$ such that for $v\not\in S$, we have $\G_\beta\in\Ocal_{v, L}\llp w\rrp$. Invoking Lemma~\ref{lem:convergence},  choose a finite set of primes $S\subseteq M_L$ such that the for $v\not\in S$, the series for $P$ and each $a_i$ have coefficients in $\Ocal_{v, L}$, and such that $\alpha$ (the lead coefficient of the series for $P$), $\beta$ (the lead coefficient of $a_d$), and $d$ are units in $\Ocal_{v, L}$.  If $v\not\in S$, it follows from the fact that $a_d\in \Ocal_{v, L}\llp w\rrp$ and $\beta\in \Ocal_{v, L}^*$ that $a_d^{-1}\in \Ocal_{v, L}\llp w\rrp$.  It is now clear that $\xi_N\in\Ocal_{v, L}\llbracket w\rrbracket$ for all $N$, and the leading coefficient of $\xi_N$ is $\alpha^{d^N}\in\Ocal_{v, L}^*$.  Now, the Hensel's Lemma construction of $G_N$ is as follows.  For $\Phi_N(X)=X^{d^N}-\xi_N$, we let $X_0=\alpha$, and then
\[X_{n+1}=X_n-\frac{\Phi(X_n)}{\Phi'(X_n)}.\]
Hensel's Lemma shows that the $X_n$ converge $\mf$-adically, and $G_N$ is their limit.
Since $X_0=\alpha\in \Ocal_{v, L}\llbracket w\rrbracket$, suppose that $X_n\in \Ocal_{v, L}\llbracket w\rrbracket$.  Then $\Phi'(X_n)=d^N X_n^{d^N-1}$ has the leading term $d^N\alpha^{d^N-1}\in \Ocal_{v, L}^*$, and hence $\Phi'(X_n)$ is a unit in $\Ocal_{v, L}\llbracket w\rrbracket$.  In other words, $X_n\in \Ocal_{v, L}\llbracket w\rrbracket$ implies $X_{n+1}\in \Ocal_{v, L}\llbracket w\rrbracket$.  But $\Ocal_{v, L}\llbracket w\rrbracket \subseteq L\llbracket w\rrbracket$ is closed in the $\mf$-adic topology, and so $G_N=\lim_{n\to \infty}X_n\in\Ocal_{v, L}\llbracket w\rrbracket$.  Similarly, $G_\infty=\lim_{N\to\infty} G_N$ has coefficients in $\Ocal_{v, L}$, and hence so too does $\G_\beta$.
\end{proof}


At this point, the power series $\G_\beta$ is simply a formal limit.  It is easy enough to see, if $L=\CC$, say, that this sort of formal convergence of functions in the local ring $\widehat{\Ocal}_{\beta, X}$ neither implies, nor is implied by, uniform convergence as functions in a neighbourhood of $\beta$.    On the other hand, it is easy to show that \emph{if} a sequence should converge both uniformly and formally, then the two limits must be identical, since a derivative of a uniform limit is the limit of the derivatives.  Before showing that $\G_\beta$ does, in fact, define a smooth function at $\beta$, we will prove a version of the Schwarz Lemma.

\begin{lemma}[Schwarz Lemma]
Let $v\in M_L$, suppose that $\epsilon>0$, and if $v$ is non-archimedean, that there is an $\alpha\in L^*$ with $\epsilon=|\alpha|_v$.   If the  series $g\in L\llbracket w\rrbracket$ converges uniformly on $U=\{w\in \overline{L}_v:|w|_v<\epsilon\},$ and $|g(w)|_v\leq B$ for all $w\in U$, then
\[|g(w)|_v\leq |w|_v^{\ord_0(g)}\frac{B}{\epsilon^{\ord_0(g)}}\]
for $w\in U$.
\end{lemma}

\begin{proof}
There is nothing to prove if $g(0)\neq 0$, so suppose that $n=\ord_0(g)\geq 1$.  We have
$g(w)=w^nh(w)$ for some $h$ with $h(0)\neq 0$, and $h$ analytic on $U$.  For any $r\in v(L)$ with $0<r<\epsilon$ we have, by the maximum modulus principle (for the non-archimedean maximum principle, see  \cite[p.~318]{robert}),
\[\max_{|w|_v\leq r}|h(w)|_v= \max_{|w|_v=r} \left|\frac{g(w)}{w^n}\right|_v\leq r^{-n} B.\]
  Since this is true for all $r<\epsilon$, we in fact have $|h(w)|_v\leq \epsilon^{-n}B$ for all $w\in U$, and hence $|g(w)|_v\leq |w|^n_v\epsilon^{-n} B$.
\end{proof}

\begin{lemma}\label{lem:analytic bottcher}
Maintaining the notation above, there is an $M_L$-divisor $\mathfrak{e}$ such that the \emph{a priori} formal power series $\G_\beta$ defines a $v$-adic analytic function on $\ann{v}{\beta}{0}{\mathfrak{e}(v)}$ for each $v\in M_L$.  Furthermore, for any $0<\delta_1<\delta_2<\mathfrak{e}(v)$, 
we have
\[\left(f^N(P)a_d^{-(d^N-1)/(d-1)}\right)^{1/d^N}\to \G_\beta\]
uniformly on $\ann{v}{\beta}{\delta_1}{\delta_2}$.
\end{lemma}

\begin{proof}  Since $P\in \basin^0_\beta(f)$, we have
\[-(d-i)\ord_{\beta}(P)>-\ord_\beta(a_i/a_d),\]
for each $0\leq i<d$, and so we may choose a disk $\ann{v}{\beta}{0}{\epsilon}$ on which
\[|P_t|_v>2_v\left|\frac{a_i(t)}{a_d(t)}\right|_v^{1/(d-i)},\]
where as usual $2_v=2$ if $v$ is archimedean, and $1$ otherwise.  For all but finitely many places, this disk can be chosen to have radius one, since we can choose a finite set of places outside of which $P_t^{d-i}a_d(t)/a_i(t)$ is given by a power series with integral coefficients, a leading coefficient which is a unit, and a pole at zero.  Proceeding similarly with $2_v|a_d(t)|_v^{-2/(d-1)}$ and the constant function $2_v$, we see that we may choose an $M_k$-divisor $\mathfrak{e}$ such that  $P_t\in \basin_v^0(f_t)$ for all $t\in \ann{v}{\beta}{0}{\mathfrak{e}(v)}$.

 Also note that since $\ord_\beta(P)=-m$, we may assume without loss of generality that the function $Pw_\beta^m$ is analytic and bounded on the disk of radius $\mathfrak{e}(v)$, say
\[|P_tw_\beta(t)^m|_v\leq  \mathfrak{b}(v),\]
where $\mathfrak{b}$ is an $M_k$-divisor.
In particular, if we set
 \[\xi_N=w_\beta^{md^N}f^N(P)a_d^{-(d^N-1)/(d-1)}\in\Ocal_{\beta, X}^*\] as above, then for any given place $v\in M_L$,  $\xi_N$ extends to an analytic function with no zeros on $\disc{v}{\beta}{\mathfrak{e}(v)}$, and hence (since $\xi_N(\beta)=\alpha^{d^N}$) there is an analytic function $G_N$ with $G_N(\beta)=\alpha$ and $G_N^{d^N}=\xi_N$.  It is not hard to see that $G_N$ is defined by the formal power series $G_N$ in the previous proof.  If $\mathfrak{c}(v)=\frac{3}{2}$ for $v$ archimedean, and $1$ otherwise, we have for $t\in \disc{v}{\beta}{\mathfrak{e}(v)}$
 \begin{eqnarray*}
 |G_N(t)|_v&=&|\xi_N(t)|_v^{1/d^N}\\
 &=& |w_\beta(t)|_v^m|f^N_t(P_t)|_v^{1/d^N} |a_d(t)|_v^{-(d^N-1)/d^N(d-1)}\\
 &\leq&|w_\beta(t)|^m_v \left(\mathfrak{c}(v)^{(d^N-1)/(d-1)}|P_t|_v^{d^N}\right)^{1/d^N}|a_d(t)|_v^{-(d^N-1)/d^N(d-1)}\\
 &\leq & \mathfrak{b}(v) \mathfrak{c}(v)^{(1-d^{-N})/(d-1)}|a_d(t)^{-1}|_v^{(1-d^{-N})/(d-1)}\\
 &\leq & \mathfrak{b}(v) \mathfrak{c}(v)^{1/(d-1)}\max\{1, |a_d(t)^{-1}|_v\}^{1/(d-1)}.
 \end{eqnarray*}
 First, we treat the case $\ord_\beta(a_d)\leq 0$.  In this case, we have for some $M_k$-divisor $\mathfrak{d}_v$,
 $|G_N(t)|_v\leq  \mathfrak{d}_v$
 for all $t\in U$, and hence
 \[|G_M(t)-G_N(t)|_v\leq 2_v\mathfrak{d}_v.\]
Now, by the Schwarz Lemma, we have
\[|G_M(t)-G_N(t)|_v\leq |w_\beta(t)|^{\ord_\beta(G_N-G_M)}_v\frac{2_v\mathfrak{d}_v}{\mathfrak{e}(v)^{\ord_\beta(G_N-G_m)}}\]
for all $t\in\disc{v}{\beta}{\mathfrak{e}(v)}$.  In particular, if $\delta_2<\mathfrak{e}(v)$, we have
\[|G_M(t)-G_N(t)|_v\leq 2_v\mathfrak{d}_v\left(\frac{\delta_2}{\mathfrak{e}(v)}\right)^{\ord_\beta(G_N-G_M)}\]
on the disk $\disc{v}{\beta}{\delta_2}$.  We've seen that $\ord_\beta(G_N-G_M)\to\infty$ as $\min\{N, M\}\to\infty$, and so the sequence $G_N$ is uniformly Cauchy on this disk.  In particular, we have $G_N\to G_\infty$ uniformly on this domain.  

Now, consider the case where $a_d(0)=0$.  Shrinking $\mathfrak{e}$ if necessary, we may assume that $|a_d(t)^{-1}|_v\geq 1$ for all $t\in \disc{v}{\beta}{\mathfrak{e}(v)}$, and so for some $M_L$-divisor $\mathfrak{d}$, we have
\[|G_M(t)-G_N(t)|_v\leq 2_v\mathfrak{d}_v|a_d(t)|_v^{-1/(d-1)}\]
on $\disc{v}{\beta}{\mathfrak{e}(v)}$.  Applying the Schwarz Lemma to $a_d(G_N-G_M)^{(d-1)}$, we find that
\[|G_M(t)-G_N(t)|_v\leq 2_v\mathfrak{d}_v|w_\beta(t)|^{\ord_\beta(G_N-G_M)+\frac{1}{d-1}\ord_\beta(a_d)}|a_d(t)|^{-1/(d-1)}.\]
Supposing that $w_\beta^{\ord_\beta(a_d)}a_d^{-1}$ is bounded on $\disc{v}{\beta}{\mathfrak{e}(v)}$ by $\mathfrak{f}(v)$, we have for all $t$ with $|w_\beta(t)|_v< \delta_2$,
\[|G_M(t)-G_N(t)|_v\leq 2_v\mathfrak{d}_v\left(\frac{\delta_2}{\mathfrak{e}(v)}\right)^{\ord_\beta(G_N-G_M)}\mathfrak{f}(v)^{\frac{1}{d-1}}.\]
Again we see that $G_M-G_N\to 0$ uniformly as $\min\{N, M\}\to\infty$, and so $G_N\to G_\infty$ uniformly on $\disc{v}{\beta}{\delta_2}$.

Since $w_\beta$ is analytic on the annulus $\ann{v}{\beta}{\delta_1}{\delta_2}$,
for any $0<\delta_1<\delta_2$,
and since $\G_\beta=w_\beta^{-m}G_\infty$, we see that $\G_\beta$ is analytic on this annulus, and is the uniform limit of
\[w_\beta^{-m}G_N=\left(f^N(P)a_d^{-(d^N-1)/(d-1)}\right)^{1/d^N}.\]
Since $\G_\beta$ is analytic on any annulus of this form, it is analytic on all of $\ann{v}{\beta}{0}{\mathfrak{e}(v)}$.

\end{proof}

%
%


\section{Proof of Theorem~\ref{th:main}}

First of all, we dispatch the somewhat pathological case where $D=D(f, P)=0$.  Note that Theorem~\ref{th:variation} says nothing at all in this case, since $D$ is of empty support.

\begin{lemma}\label{lem:D=0}
Theorem~\ref{th:main} holds in the case where $D(f, P)=0$.
\end{lemma}

\begin{proof}
One possible case in which $D(f, P)=0$ is the case in which $P$ is \emph{preperiodic} for $f$, that is, the case where $f^m(P)=f^n(P)$ for some $m>n\geq 0$.  In this case, $f_t^m(P_t)=f^n_t(P_t)$ for all $t\in X(\overline{k})$, and so the set $\{f^N_t(P_t):N\geq 0\}$ is finite, for each $t\in X(\overline{k})$.  It follows immediately that
\[\h_{f_t}(P_t)=\lim_{N\to\infty}d^{-N}h(f_t^N(P_t))=0\] identically on $X(\overline{k})$.  The inequality \eqref{eq:main bound}, in this case, is trivial.

Suppose that $P$ is not preperiodic for $f$, but that $D(f, P)=0$, and hence $\h_f(P)=0$.  By a theorem of Benedetto \cite{benedetto}, the polynomial $f$ is \emph{(affine) isotrivial}.  Thus, there exists an affine transformation
\[\psi(z)=\alpha z+ \beta\]
with $\alpha\neq 0$, and $\alpha, \beta\in \overline{K}$, such that
$\psi\circ f\circ\psi^{-1} \in \overline{k}[z].$  In other words, there is a dominant morphism $\phi:Y\to X$ defined over $\overline{k}$,  $\alpha, \beta\in\overline{k}(Y)$, and $g\in \overline{k}[z]$ such that
\[f_{\phi(s)}(z)=\psi_s^{-1}\circ g\circ\psi_s(z).\]
If we let $Q=\psi(P\circ\phi)\in \overline{k}(Y)$, and fix any $\gamma\in Y(\overline{k})$, then
\begin{eqnarray}
\ord_\gamma(g^N(Q))&=&\ord_\gamma(\alpha (f^N(P)\circ\phi)+\beta)\nonumber\\ &\geq&\min\left\{\ord_\gamma(\alpha)+\ord_\gamma((f^N(P))\circ\phi), \ord_\gamma(\beta)\right\}\label{eq:pullback}\\
&=&\min\left\{\ord_\gamma(\alpha)+e_\gamma(\phi)\ord_{\phi(\gamma)}(f^N(P)), \ord_\gamma(\beta)\right\}\nonumber
\end{eqnarray}
with equality in \eqref{eq:pullback} if the two terms in the minimum are distinct (here $e_\gamma(\phi)$ is the ramification index of $\phi$ at $\gamma$).  If $\hl_{f, \phi(\gamma)}(P)>0$, then $\ord_{\phi(\gamma)}(f^N(P))$ decreases without bound as  $N\to\infty$.  It follows that for $N$ sufficiently large we have 
\[\ord_\gamma(g^N(Q))=\ord_\gamma(\alpha)+e_\gamma(\phi)\ord_{\phi(\gamma)}(f^N(P)),\]
and hence
\[\hl_{g, \gamma}(Q)=e_\gamma(\phi)\hl_{f, \phi(\gamma)}(P).\]
On the other hand, if $\hl_{f, \phi(\gamma)}(P)=0$, then $\ord_{\phi(\gamma)}(f^N(P))$ is bounded as $N\to\infty$, and so $\ord_\gamma(g^N(Q))$ is bounded as well; it follows that $\hl_{g, \gamma}(Q)=0$.  In other words, we have shown that
\[D(g, Q)=\sum_{\gamma\in Y(\overline{k})}\hl_{g, \gamma}(Q)(\gamma)=\sum_{\beta\in X(\overline{k})} \hl_{f, \beta}(P)\left(\sum_{\gamma\in\phi^{-1}(\beta)}e_\gamma(\phi)(\gamma)\right)=\phi^* D(f, P).\]
This is true in general, but in particular $D(f, P)=0$ implies $D(g, Q)=0$.  It is easy to see that if $g\in \overline{k}[z]$, then $D(g, Q)=Q^*(\infty)$, and so $D(g, Q)=0$ implies that $Q$ is constant.

Now, for each fixed $s\in Y(\overline{k})$, $\psi_s^{-1}:\PP^1\to\PP^1$ is a morphism of degree 1, and so
\[h\left(f^N_{\phi(s)}(P_{\phi(s)})\right)=h\left(\psi_s^{-1}\circ g^N(Q)\right)=h(g^N(Q))+O(1),\]
where the implied constant depends on $s$, but not on $N$.  Dividing by $d^N$ and letting $N\to\infty$, we have
$\h_{f_t}(P_t)=\h_g(Q)$
for all $t\in X(\overline{k})$, since $\phi$ was dominant, and so  $\h_{f_t}(P_t)$ is constant.  Since $h_D=0$, \eqref{eq:main bound} holds.
\end{proof}

We now prove a lemma which contains most of the content of Theorems~\ref{th:main} and~\ref{th:variation}.  We set up the notation as above, with $L$ a field, $v\in M_L$ some valuation, $X/L$ a smooth and projective curve, $f(z)\in L(X)[z]$, and $P\in L(X)$.  Furthermore, in light of Lemma~\ref{lem:D=0}, we will suppose that $D(f, P)\neq 0$, whereupon \[\h_f(P)=\deg(D(f, P))>0.\]  
\begin{lemma}\label{lem:main}
There is an $M_L$-divisor $\mathfrak{b}$ such that
\[\left|\hl_{f_t, v}(P_t)-\lambda_{D, v}(t)\right|\leq\log\mathfrak{b}(v)\]
for all $t\in X(\overline{L}_v)$ and all $v\in M_L$ (in particular, the difference vanishes identically at all but finitely many places).  Furthermore there is an integer $N$ such that for each $\beta\in \Supp(D)$ there is a germ $E_\beta\in \widehat{\Ocal}_{\beta, X}$, and an $M_L$-divisor $\mathfrak{e}$ such that $E_\beta$ is $v$-adic analytic on $\disc{v}{\beta}{\mathfrak{e}(v)}$, and 
\[\hl_{f_t, v}(P_t)-\lambda_{D, v}(t)=\frac{1}{d^N(d-1)}\log|E_\beta(t)|_v\]
on $\ann{v}{\beta}{0}{\mathfrak{e}(v)}$.
\end{lemma}

\begin{proof}
Let $N$ be the integer chosen in Lemma~\ref{lem:r-r}, which we may take as the least non-negative integer with the property that \[d^N(d-1)\deg(D(f, P))\geq 2g(X)\] and $f^N(P)\in\basin_\beta^0(f)$ for all $\beta\in \Supp(D)$.  We have \[\hl_{f_t, v}(f^N(P_t))=d^N \hl_{f_t, v}(P_t)\] and (by definition) 
\[\lambda_{D(f, f^N(P)), v}(t)=d^N \lambda_{D(f, P), v}(t),\]
and so the general case clearly follows from the special case where $N=0$.  Consequently, we will suppose throughout that $N=0$.

To begin, set $Z\subseteq X(\overline{L})$ to be a finite set of points containing all of the poles of $P$ and of the $a_i$, and at each $\beta\in Z$ we fix a uniformizer $w_\beta\in L(X)$.  Recall that, for each $\beta$, we have an inclusion $L(X)\hookrightarrow L\llp w_\beta\rrp$, and we will associate functions with their images (their Laurent series).  

Choose a $\beta\in\Supp(D)$.  By Lemma~\ref{lem:analytic bottcher}, there is an $M_L$-divisor $\mathfrak{e}_\beta$ such that the formal limit $\G_\beta\in L\llp w_\beta \rrp$ of $(f^N(P)a_d^{-(d^N-1)/(d-1)})^{1/d^N}$ defines a $v$-adic analytic function on $\ann{v}{\beta}{0}{\mathfrak{e}_\beta(v)}$
with a pole of order $m=-\ord_\beta(P)$ at $\beta$. For simplicity, and since $\beta$ is fixed, we will drop the subscripts.

Now, as in the proof of Lemma~\ref{lem:analytic bottcher}, we may suppose that $\mathfrak{e}(v)$ is small enough that $P_t\in\basin_v^0(f_t)$ for all $t\in\ann{v}{\beta}{0}{\mathfrak{e}(v)}$.  By the definition of $\basin_v^0(f_t)$, and the fact that this set is closed under $f_t$, this implies $|f_t^N(P_t)|_v>1$ for all $N$, and so in particular,
\[\hl_{f_t, v}(P_t)=\lim_{N\to\infty}d^{-N}\log|f^N_t(P_t)|_v\]
for $t\in \ann{v}{\beta}{0}{\mathfrak{e}(v)}$.  On the other hand, for each $t\in \ann{v}{\beta}{0}{\mathfrak{e}(v)}$ we have by Lemma~\ref{lem:analytic bottcher}
\begin{eqnarray*}\log|\G(t)|_v&=&\log\lim_{N\to\infty}\left|f^N_t(P_t)a_d(t)^{-(d^N-1)/(d-1)}\right|_v^{1/d^N}\\
&=&\lim_{N\to\infty}d^{-N}\log|f^N_t(P_t)|_v-\frac{1}{d-1}\log|a_d(t)|_v\\
&=&\hl_{f_t, v}(P_t)-\frac{1}{d-1}\log|a_d(t)|_v.
\end{eqnarray*}
 Since $P\in\basin^0_\beta(f)$, \[(d-1)\hl_{f, \beta}(P)=(d-1)\log|P|_\beta+\log|a_d|_\beta,\] and so the function $g$ defining the local heights has a pole of this order  at $\beta$.  Shrinking $\mathfrak{e}$ again, if necessary, we suppose that $\mathfrak{e}$ is also a radius of convergence for the series defining $g$, and that $|g(t)|_v\geq 1$ for all $t\in\ann{v}{\beta}{0}{\mathfrak{e}(v)}$.
Thus, for $t$ in this domain,
\[\lambda_{D, v}(t)=(d-1)^{-1}\log|g(t)|_v,\]
whereupon
\begin{eqnarray}\label{eq:difference local heights}
\hl_{f_t, v}(P_t)-\lambda_{D, v}(t)&=&\log|\G(t)|_v+\frac{1}{d-1}\log|a_d(t)|_v-\frac{1}{d-1}\log|g(t)|_v\\
&=&\frac{1}{d-1}\log|\G(t)^{d-1}a_d(t)/g(t)|_v\nonumber
\end{eqnarray}
Now, if $E_\beta=\G^{d-1}a_d/g\in K_\beta$, then $\ord_\beta(E_\beta)=0$, and hence $E_\beta\in\widehat{\Ocal}_{\beta, X}$ has a power series representation of the form
\[E_\beta=b_0+b_1w+b_2w^2+\cdots\] 
with $b_i\in L$ and $b_0\neq 0$.  Since $\G$, $a_d$, and $g$ are $v$-adic analytic functions on $\ann{v}{\beta}{0}{\mathfrak{e}(v)}$, and $g$ has no zeros in this region, $E_\beta$ is also $v$-adic analytic.  This is the function $E_\beta$ in the statement of the lemma.

Note that, by Lemmas~\ref{lem:convergence} and \ref{lem:formal bottcher}, there is a finite set $S\subseteq M_L$ of places (containing all infinite places) such that for $v\not\in S$, we have $\G,a_d, g\in\Ocal_{v, L}\llp w\rrp$, and the leading coefficients of $\G$, $a_d$, and $g$ are in $\Ocal_{v, L}^*$.  It follows that for $v\not\in S$, we have $E_\beta\in\Ocal_{v, L}\llbracket w\rrbracket$, and $E_\beta(\beta)=b_0\in\Ocal_{v, L}^*$. 
  Now, if $v\not\in S$, if $\mathfrak{e}(v)=1$,  and if $t\in\ann{v}{\beta}{0}{\mathfrak{e}(v)}$, then
\begin{eqnarray*}
\hl_{f_t, v}(P_t)-\lambda_{D, v}(t)&=&(d-1)^{-1}\log|b_0+w(t)(b_1+b_2w(t)+\cdots)|_v\\
&=&(d-1)^{-1}\log|b_0|_v=0.
\end{eqnarray*}
At the other finitely many places, we may shrink $\mathfrak{e}(v)$ to ensure that $E_\beta$ has no zeros on $\disc{v}{\beta}{2\mathfrak{e}(v)}_v$.  This ensures that $\log|E_\beta(t)|_v$ is bounded above and below for $t\in\disc{v}{\beta}{\mathfrak{e}(v)}$, and hence so is $\hl_{f_t, v}(P_t)-\lambda_{D, v}(t)$.  So, we may choose an $M_L$-divisor $\mathfrak{d}_\beta$ such that $t\in\disc{v}{\beta}{\mathfrak{e}_\beta(v)}$ implies
\[\left|\hl_{f_t, v}(P_t)-\lambda_{D, v}(t)\right|=(d-1)^{-1}\left|\log|E_\beta(t)|_v\right|\leq \log\mathfrak{d}_{\beta}(v).\]
We construct $\mathfrak{e}_\beta$ and $\mathfrak{d}_\beta$ in this way, for each $\beta\in\Supp(D)$, and set
\[\mathfrak{d}_1=\max_{\beta\in\Supp(D)}\mathfrak{d}_\beta.\]

Now let $\beta\in Z\setminus\Supp(D)$, and again choose $\mathfrak{e}=\mathfrak{e}_\beta$ which is a global radius of convergence for all of the series representing $P$, $a_i$, and $g$ in $L\llp w_\beta\rrp$.  As above, we choose a set of places $S\subseteq M_L$ large enough that for any $v\not\in S$, $P$, the $a_i$, and $g$ are in $\Ocal_{v, L}\llp w_\beta\rrp$.  We will enlarge $S$, if necessary, to contain the finitely many places $v\in M_L$ with $\mathfrak{e}(v)\neq 1$.  Since $f^N(P)$ is a polynomial in $P$ and the $a_i$, it is clear that each of these is defined by a series in $\Ocal_{v, L}\llp w\rrp$, for $v\not\in S$, convergent within the same radius $\mathfrak{e}$.  Now, for each $N$, let $m_N=-\ord_\beta(f^N(P))$.

  First, consider $v\not\in S$, so that $\mathfrak{e}_\beta(v)=1$.  Then for $t\in \ann{v}{\beta}{0}{\mathfrak{e}(v)}$, we have
\[\left|w(t)^{m_N}f^N_t(P_t)\right|_v\leq 1,\] since $w^{m_N}f^N(P)$ is regular at $\beta$, and defined by a series in $w$ with coefficients which are integral at $v$.
Now let $\delta>0$ be a real number.  For $t\in\ann{v}{\beta}{\delta}{\mathfrak{e}(v)}$, we have
\begin{eqnarray*}
d^{-N}\log|f^N_t(P_t)|&=& d^{-N}\left(\log \left|w(t)^{-m_N}\right|_v+\log\left|w(t)^{m_N}f^N_t(P_t)\right|_v\right)\\
&\leq& d^{-N}m_N\log \delta^{-1}.
\end{eqnarray*}
But $\beta\not\in\Supp(D)$, and so by definition, we have
\[d^{-N}m_N=d^{-N}\max\left\{0, -\ord_\beta(f^N(P))\right\}\longrightarrow 0\]
as $N\to \infty$.  In particular, $\hl_{f_t, v}(P_t)=0$ whenever $t\in \ann{v}{\beta}{\delta}{\mathfrak{e}(v)}$.  But $\delta>0$ was arbitrary, and so $\hl_{f_t, v}(P_t)=0$ for all $t\in\ann{v}{\beta}{0}{\mathfrak{e}(v)}$ whenever $v\not\in S$.

We now consider the remaining finitely many places.  Fix $v$, and fix a  real number $0<\delta<\mathfrak{e}(v)$.  Shrinking $\mathfrak{e}$ if necessary, we can assume that the power series for $P$, the $a_i$, and $g$ all converge within some radius strictly greater than $\mathfrak{e}(v)$.  If $v$ is archimedean,
then the maximum modulus principle ensures that, for $t$ in $\disc{v}{\beta}{\mathfrak{e}(v)}$,
\[|w(t)^{m_N}f^N_t(P_t)|_v\leq \max_{|w(t)|_v=\mathfrak{e}(v)}\left|w(t)^{m_N}f^N_t(P_t)\right|_v=\mathfrak{e}_\beta(v)^{m_N}\max_{|w(t)|_v=\mathfrak{e}(v)}\left|f^N_t(P_t)\right|_v.\]  The same follows for non-archimedean valuations by the non-archimedean maximum modulus principle \cite[p.~318]{robert}, given that our definition of an $M_L$-divisor required that $\mathfrak{e}(v)=|\alpha|_v$ for some $\alpha\in L^*$.
Fixing $v$ for the moment, define
\[\|j\|=\max_{|w(t)|_v=\mathfrak{e}(v)}|j(t)|_v,\]
for any $j\in L(X)$ for which this maximum exists, and define $\Phi(X)\in \RR[X]$ by 
\[\Phi(X)=\sum_{i=0}^d \|a_i\| X^i.\]
Note that the triangle inequality gives
\[\|f^{N+1}(P)\|\leq \Phi(\|f^N(P)\|)\]
On the other hand, $\Phi$ is a real polynomial with non-negative coefficients, and so is non-decreasing on positive values.  Thus, for all $N$,
$\|f^{N}(P)\|\leq \Phi^N(\|P\|)$ and so,
by Lemma~\ref{lem:basins}, there is a $B_v$ (which depends on $\|P\|$) such that
\[\|f^N(P)\|\leq \Phi^N(\|P\|)\leq B_v^{d^N}\]
for all $N$.
If we suppose that $t\in \ann{v}{\beta}{\delta}{\mathfrak{e}(v)}$, then we have
\begin{eqnarray*}
|f^N_t(P_t)|_v&=&|w(t)^{-m_N}|_v|w(t)^{m_N}f^N_t(P_t)|_v\\
&\leq& \left(\frac{\mathfrak{e}_\beta(v)}{\delta}\right)^{m_N}\|f^N(P)\|\\
&\leq& \left(\frac{\mathfrak{e}_\beta(v)}{\delta}\right)^{m_N}B_v^{d^N}.
\end{eqnarray*}
Since $m_N$ is bounded as $N\rightarrow\infty$, taking logarithms and limits gives
\[\hat{\lambda}_{f_t, v}(P_t)\leq \log B_v,\]
whenever $t\in\ann{v}{\beta}{\delta}{\mathfrak{e}(v)}$.  As $\delta$ was arbitrary, and $B_v$ did not depend on $\delta$, we have
$\hat{\lambda}_{f_t, v}(z_t)\leq\log B_v$ for all $t\in\ann{v}{\beta}{0}{\mathfrak{e}(v)}$.  Let $\mathfrak{n}_\beta$ be the $M_k$-divisor with value $B_v$ at each of these places, and 1 everywhere else, so that for any $v\in M_L$, 
\[\hl_{f_t, v}(P_t)\leq \log\mathfrak{n}(v)\]
whenever $t\in\ann{v}{\beta}{0}{\mathfrak{e}(v)}$.

On the other hand, since the poles of the function $g$ are all contained in $\Supp(D)$, there is, by Lemma~\ref{lem:cont}, an $M_L$-divisor $\mathfrak{m}_1$ such that 
$\left|g(t)\right|_v\leq \mathfrak{m}_1(v)$
whenever $|w_{\beta'}(t)|_v\geq \mathfrak{e}_{\beta'}(v)$ for all $\beta'\in\Supp(D)$.  In other words, either $t\in\disc{v}{\beta'}{\mathfrak{e}_{\beta'}(v)}$, for some $\beta'\in\Supp(D)$, in which case we have
\[\left|\hat{\lambda}_{f_t, v}(z_t)-\lambda_{D, v}(t)\right|\leq \log\mathfrak{d}_{1}(v),\]
or else $\lambda_{D, v}(t)\leq \log\mathfrak{m}_1(v)$.  In the latter case, if $t\in\ann{v}{\beta}{0}{\mathfrak{e}_\beta(v)}$, then  we have shown that $\hat{\lambda}_{f_t, v}(P_t)\leq \log \mathfrak{n}_\beta(v)$.  Combining these, we have
\[\left|\hat{\lambda}_{f_t, v}(z_t)-\lambda_{D, v}(t)\right|\leq \log\mathfrak{m}_1(v)+\log\mathfrak{n}_\beta(v).\]
Constructing $\mathfrak{e}_\beta$ and $\mathfrak{n}_\beta$ as above, for each $\beta\in Z\setminus\Supp(D)$, and letting
\[\mathfrak{d}_2(v)=\max_{\beta\in Z\setminus\Supp(D)}\{\mathfrak{d}_1(v), \mathfrak{m}_1(v)\mathfrak{n}_\beta(v)\},\]
then, we have
\[\left|\hat{\lambda}_{f_t, v}(z_t)-\lambda_{D, v}(t)\right|\leq \log \mathfrak{d}_2(v)\]
whenever $t\in\disc{v}{\beta}{\mathfrak{e}_\beta(v)}$,  for \emph{some} $\beta\in Z$.

For each $\beta\in Z$ we have chosen an $M_L$-divisor $\mathfrak{e}_\beta$.  Since none of the functions $P$, $a_i$, and $g$ have any poles outside of $Z$,  we employ Lemma~\ref{lem:cont} to construct an $M_L$-divisor $\mathfrak{m}_2$ such that for each $v\in S$, the condition $|j(t)|_v>\mathfrak{m}_2(v)$ for \emph{any} of the functions $j\in\{P, a_i, g\}$, implies $|w_\beta(t)|_v<\mathfrak{e}_\beta(v)$ for some $\beta\in Z$.  For any place $v$, let $Y_v\subseteq X(\overline{L}_v)$ be the set of points $t$ such that \[\max\{|P_t|_v, |a_i(t)|_v, |g(t)_v\}\leq\mathfrak{m}_2(v).\]
Then  we have, as above, $M_L$-divisors $\mathfrak{m}_3$ and $\mathfrak{m}_4$ such that for $t\in Y_v$,
\[\hat{\lambda}_{f_t, v}(P_t)\leq \log \mathfrak{m}_3(v)\quad\text{ and }\quad\lambda_{D, v}(t)\leq \log \mathfrak{m}_4(v).\]
That is,
\[\left|\hat{\lambda}_{f_t, v}(P_t)-\lambda_{D, v}(t)\right|\leq \log\mathfrak{d}_3(v),\]
for $\mathfrak{d}_3=\mathfrak{m}_3\mathfrak{m}_4$.
On the other hand, if $t\not\in Y_v$, then $|w_\beta(t)|_v<\mathfrak{e}_\beta(v)$ for some $\beta\in Z$.  By the previous two arguments, we have
\[\left|\hat{\lambda}_{f_t, v}(P_t)-\lambda_{D, v}(t)\right|\leq \log \mathfrak{d}_i(v)\]
for $i=1$ or  $2$.  Letting $\mathfrak{d}_4=\max_{1\leq i\leq 3}\mathfrak{d}_i$, pointwise, we have
\[\left|\hat{\lambda}_{f_t, v}(P_t)-\lambda_{D, v}(t)\right|\leq \log \mathfrak{d}_4(v)\]
for all $t\in X(\overline{L}_v)$.
\end{proof}

Theorem~\ref{th:main} is an immediate consequence of the above lemma. 
\begin{proof}[Proof of Theorem~\ref{th:main}]
Let $k$, $X$, $K$, $f$, and $P$ be as in the statement of Theorem~\ref{th:main}.  Then, if $\mathfrak{b}$ be the $M_k$-divisor prescribed by Lemma~\ref{lem:main}, and $L/k$ is any finite Galois extension,  we have for any $t\in X(L)$,
\[\h_{f_t}(P_t)=\sum_{v\in M_k}\frac{[k_v:\QQ_v]}{[k:\QQ]}\left(\frac{1}{[L:k]}\sum_{\sigma\in\Gal(L/k)}\hl_{f_{t^\sigma}, v}(P_{t^\sigma})\right)\]
and
\[h_D(t)=\sum_{v\in M_k}\frac{[k_v:\QQ_v]}{[k:\QQ]}\left(\frac{1}{[L:k]}\sum_{\sigma\in\Gal(L/k)}\lambda_{D, v}(t^\sigma)\right),\]
and hence
\begin{eqnarray*}
\left|\h_{f_t}(P_t)-h_D(t)\right|
&=&\left|\sum_{v\in M_k}\frac{[k_v:\QQ_v]}{[k:\QQ]}\sum_{\sigma\in\Gal(L/k)}\frac{\hl_{f_{t^\sigma}, v}(P_{t^\sigma})-\lambda_{D, v}(t^\sigma)}{[L:k]}\right|\\
&\leq &\sum_{v\in M_k}\frac{[k_v:\QQ_v]}{[k:\QQ]}\sum_{\sigma\in\Gal(L/k)}\frac{\left|\hl_{f_{t^\sigma}, v}(P_{t^\sigma})-\lambda_{D, v}(t^\sigma)\right|}{[L:k]}\\
&\leq & \sum_{v\in M_k}\frac{[k_v:\QQ_v]}{[k:\QQ]}\sum_{\sigma\in\Gal(L/k)}\frac{\log\mathfrak{b}(v)}{[L:k]}\\
&=&\sum_{v\in M_k}\frac{[k_v:\QQ_v]}{[k:\QQ]}\log \mathfrak{b}(v).
\end{eqnarray*}
This final sum is finite, and independent of both $t$ and $L$.  Since $L/k$ was arbitrary, we have
\[\h_{f_t}(P_t)=h_D(t)+O(1),\]
for $t\in X(\overline{k})$.
\end{proof}

To prove the corollary, we use an argument due to Lang.
\begin{proof}[Proof of Corollary~\ref{cor:main}]
Let $\eta\in\Pic(X)$ have degree 1, and let $h_\eta$ be a height relative to this divisor.
By the linearity of heights,
\[h_{D(f, P)}-\h_f(P)h_\eta=h_{D(f, P)-\h_f(P)\eta}+O(1).\]
Note that the divisor $D(f, P)-\h_f(P)\eta$ has degree 0.
In general, by \cite[Proposition 5.4, p.~115]{lang}, we have
\[h_{D(f, P)-\h_f(P)\eta}=O(h_\eta^{1/2}).\]
  If $X=\PP^1$, then $D(f, P)-\h_f(P)\eta$ is linearly equivalent to the zero divisor, and so
\[h_{D(f, P)}(t)-\h_f(P)h_\eta(t)=O(1).\]
\end{proof}


\section{Proof of Theorems~\ref{th:quad polys} and~\ref{th:variation}}

We will begin by proving Theorem~\ref{th:variation}.  Theorem~\ref{th:quad polys} follows by clarifying some of the details of this proof in the special case where $f_t(z)=z^2+t$, and $P_t\in\ZZ[t]$.
\begin{proof}[Proof of Theorem~\ref{th:variation}]
Let $k$, $X$, $f$, and $P$ be as in the theorem.  By Lemma~\ref{lem:main}, there is a finite set of places $S\subseteq M_k$, containing all infinite places, such that for $v\not\in S$, we have
\[\hl_{f_t, v}(P_t)=\lambda_{D, v}(t),\]
for all $t\in X(\overline{k})$.  Again by Lemma~\ref{lem:main}, we may choose, for each $\beta\in\Supp(D)$, a germ $E_\beta\in\widehat{\Ocal}_{\beta, X}$, defined over $k$, and an $M_k$-divisor $\mathfrak{e}_\beta$ such that
\[\hl_{f_t, v}(P_t)-\lambda_{D, v}(t)=\frac{1}{d^N(d-1)}\log|E_\beta(t)|_v\]
for all $t\in\disc{v}{\beta}{\mathfrak{e}(v)}$.  We will enlarge $S$, if necessary, to ensure that $E_\beta(\beta)$ is an $S$-unit for all $\beta\in \Supp(D)$.  Let $\tilde{E}_\beta(t)=E_\beta(t)/E_\beta(\beta)$, so that $\tilde{E}_\beta$ is $v$-adic analytic on $\disc{v}{\beta}{\mathfrak{e}(v)}$, and $\tilde{E}_\beta(\beta)=1$.  Since $1^{d^N(d-1)}\equiv \tilde{E}_\beta\MOD{\mf}$, for $\mf$ the maximal ideal of $\widehat{\Ocal}_{\beta, X}$, we see (by Hensel's Lemma) that there is an $\tilde{F}_\beta\in\widehat{\Ocal}_{\beta, X}$ such that
\[\tilde{F}_\beta(t)=1+O(\mf)\qquad\text{and}\qquad \tilde{F}_\beta^{d^N(d-1)}=\tilde{E}_\beta.\]
The germ $\tilde{F}_\beta$ defines a $v$-adic analytic function on $\disc{v}{\beta}{\mathfrak{e}'(v)}$, for some $M_k$-divisor $\mathfrak{e}'$ and, shrinking $\mathfrak{e}'(v)$ at finitely many places if necessary, we may assume that $|\tilde{F}_\beta(t)|_v=1$ for all $t\in\disc{v}{\beta}{\mathfrak{e}'(v)}$, whenever $v\in S$ is non-archimedean.

If $\phi:S\to \Supp(D)$, and $t\in X(k)$ satisfies $t\in \disc{v}{\phi(v)}{\mathfrak{e}'(v)}$ for each $v\in S$, then we have
\begin{eqnarray*}
\h_{f_t}(P_t)-h_D(t)&=&\sum_{v\in M_k}\frac{[k_v:\QQ_v]}{[k:\QQ]}\left(\hl_{f_t, v}(P_t)-\lambda_{D, v}(t)\right)\\
&=&\sum_{v\in S}\frac{[k_v:\QQ_v]}{[k:\QQ]}\left(\hl_{f_t, v}(P_t)-\lambda_{D, v}(t)\right)\\
&=&\sum_{v\in S}\frac{[k_v:\QQ_v]}{[k:\QQ]}\frac{1}{d^N(d-1)}\log|E_{\phi(v)}(t)|_v\\
&=&\sum_{v\in S}\frac{[k_v:\QQ_v]}{[k:\QQ]}\left(\log|\tilde{F}_{\phi(v)}(t)|_v+\frac{1}{d^N(d-1)}\log|E_{\phi(v)}(\phi(v))|_v\right)\\
&=&\sum_{v\mid\infty}\frac{[k_v:\QQ_v]}{[k:\QQ]}\log|\tilde{F}_{\phi(v)}(t)|_v+C(\phi),
\end{eqnarray*}
where
\[C(\phi)=\frac{1}{d^N(d-1)}\sum_{v\in S}\frac{[k_v:\QQ_v]}{[k:\QQ]}\log|E_{\phi(v)}(\phi(v))|_v.\]
Since $\tilde{F}_\beta$ is $v$-adic analytic, for each $v\mid \infty$, and $\tilde{F}_\beta(t)=1$, we may shrink $\mathfrak{e}'(v)$ again to ensure that $\tilde{F}_\beta(t)\neq 0$ for $t\in\disc{v}{\beta}{\mathfrak{e}'(v)}$.  It follows that
\[F_{\beta, v}(t)=\frac{[k_v:\QQ_v]}{[k:\QQ]}\log|\tilde{F}_\beta(t)|_v\]
is real analytic on $\disc{v}{\beta}{\mathfrak{e}'(v)}\subseteq X(\overline{k}_v)$.  Since $\tilde{F}_\beta(\beta)=1$, we have $F_{\beta, v}(\beta)=0$.

  If $k_v=\RR$, then since \[\tilde{F}_\beta=1+c_1w+c_2w^2+\cdots,\] the function \[F_{\beta, v}=\frac{[k_v:\QQ_v]}{[k:\QQ]}\left(c_1w+\left(c^2-\frac{c_1^2}{2}\right)w^2+\left(c_3-c_1c_2 +\frac{c_1^3}{3}\right)w^3+\cdots\right)\] is given by a power series in $w$ with coefficients in $k$.  Similarly, if $k_v=\CC$, and the disk $|w|_v<\mathfrak{e}'(v)$ is identified with a disk in $\RR^2$ by $w=x+iy$, then $F_{\beta, v}\in k\llbracket x, y \rrbracket$.

The result is easily extended to a theorem quantified over $X(\overline{k})$.  Let $L/k$ be a Galois extension, and let  $\phi:S\times \Gal(L/k)\to\Supp(D)$.  For any $t\in X(L)$ satisfying $t^\sigma\in \disc{v}{\phi(v, \sigma)}{\mathfrak{e}'(v)}$, for all $v\in S$ and $\sigma\in\Gal(L/k)$, we have
\[\h_{f_t}(P_t)-h_{D}(t)=\frac{1}{[L:\QQ]}\sum_{\substack{\sigma\in\Gal(L/k)\\v\mid\infty}}F_{\phi(v, \sigma)}(t^\sigma)+C(\phi),\]
where
\[C(\phi)=\frac{1}{d^N(d-1)}\sum_{\substack{\sigma\in\Gal(L/k)\\ v\in S}}\frac{[k_v:\QQ_v]}{[L:\QQ]}\log|E_{\phi(v, \sigma)}(\phi(v, \sigma))|\]
by the same argument as above.

\end{proof}

\begin{remark}
It should be pointed out that, in the statement of Theorem~\ref{th:variation}, it is entirely possible that for certain $\phi:S\to\Supp(D)$, there will be no $t\in X(k)$ satisfying $t\in \disc{v}{\phi(v)}{\mathfrak{e}'(v)}$ for all $v\in S$.  In particular, suppose that $k=\QQ$, $X=\PP^1$, and that $D=(i)+(-i)$, for $i^2=-1$ (this $D$ arises, for example, when $f(z)=z^2+t$ and $P_t=t^3(t^2+1)^{-1}$).  Then for $\epsilon$ small enough, there is no $t\in\PP^1(\QQ)$ with $t\in\disc{\infty}{i}{\epsilon}$.  In this case, Theorem~\ref{th:variation} is vacuously true, but becomes non-trivial after a finite extension.

\end{remark}

\begin{proof}[Proof of Theorem~\ref{th:quad polys}]
Let $f_t(z)=z^2+t$, and let $P\in \ZZ[t]$ be a polynomial of degree at least one, and leading coefficient $\alpha$.  

  Note that $\ord_\beta(P), \ord_\beta(t)\geq 0$ for all $\beta\neq\infty=[1:0]\in\PP^1$.  It follows that $\hl_{f, \beta}(P)=0$ for all $\beta\neq\infty$.  On the other hand, 
\[\ord_\infty(P)=\deg(P)\geq 1>\frac{1}{2}=\frac{1}{2}\ord_\infty(t),\]
and consequently,
 $P\in\basin^0_\infty(f)$.  It follows at once that
\[D(f, P)=\hl_{f, \infty}(P)(\infty)=\deg(P)(\infty),\]
and so $h_D$ can be taken to be $\deg(P)h$, for $h$ the usual Weil height on $\PP^1$.  Similarly, the N\'{e}ron functions can be taken to be \[\lambda_{D, v}(x)=\max\{0, \log|x^{\deg(P)}|_v\}.\]

Now, for each non-archimedean $v\in M_\QQ$, if $|t|_v\leq 1$, then $|f^N_t(P_t)|_v\leq 1$ for all $N$, since $P$ and $f$ have integral coefficients.  In this case, we have
\[\hl_{f_t, v}(P_t)=\lambda_{D, v}(t).\]
  If, on the other hand, $|t|_v>\max\{1, |\alpha|_v^{-1}\}$ then 
we have (since the coefficients of $P$ are integral)
\[|P_t|_v=|\alpha|_v|t|_v^{\deg(P)}\geq |t|_v,\]
whence $P_t\in \basin_v^0(f_t)$.  It follows from Lemma~\ref{lem:basins} that
\[\hl_{f_t, v}(P_t)=\log|\alpha t^{\deg(P)}|_v=\lambda_{D, v}(t)+\log|\alpha|_v\]
in this case.
Thus, if we consider only $t\in X(\QQ)$ such that $|t|_v>\max\{1, |\alpha|_v^{-1}\}$ for all $v\in M_k$ with $|\alpha|_v\neq 1$ 
 it follows that
\[\h_{f_t}(P_t)-h_D(t)=\hl_{f_t, \infty}(P_t)-\lambda_{D, \infty}(t)-\log|\alpha|,\]
by the product formula.
Now, just as in the proof of Theorem~\ref{th:variation}, we see that (taking $w(t)=t^{-1}$ as a uniformizer)
\[\hl_{f_t, \infty}(P_t)-\lambda_{D, \infty}(t)=\log|\G_\infty(t) t^{-\deg(P)}|,\]
where $\G_\infty(t)\in\QQ\llbracket t^{-1}\rrbracket$ is analytic in  a punctured  neighbourhood of $\infty$, and
\[\G_\infty(t)=\alpha t^{\deg(P)}+O(t^{\deg(P)-1})\]
by construction, the leading coefficient being that of $P$.  We have
\[\G_\infty(t)t^{-\deg(P)}=\alpha+O(t^{-1}),\]
and so if we choose  a neighbourhood small enough that $\G_\infty(t)t^{-\deg(P)}\neq 0$, for all $t$ in the neighbourhood, we have
\[\log|\G_\infty(t)t^{-\deg(P)}|=\log|\alpha|+F(t^{-1}),\]
for some $F(x)\in \QQ\llbracket x\rrbracket$ with $F(0)=0$.  Thus, for all $t\in X(\QQ)$ in this real neighbourhood of $\infty$, and such that $|t|_v>\max\{1, |\alpha|_v^{-1}\}$ for all $v\in M_\QQ^0$ with $|\alpha|_v\neq 1$, we have
\[\h_{f_t}(P_t)-h_D(t)=F\left(\frac{1}{t}\right).\]
Note that the second condition
 is satisfied vacuously by all $t\in X(\QQ)$ if, for example, $P$ is monic.
\end{proof}



\end{document}